\documentclass[12pt]{amsart}

\usepackage[english]{babel}
\usepackage[utf8x]{inputenc}
\usepackage[T1]{fontenc}
\usepackage{todonotes}
\usepackage{tikz}
\usepackage{mathtools}

\usepackage[a4paper,top=3cm,bottom=3cm,left=3cm,right=3cm,marginparwidth=1.75cm]{geometry}

\usepackage{varioref}
\usepackage{hyperref}
\usepackage[nameinlink, capitalize, noabbrev]{cleveref}

\usepackage{graphicx}
\usepackage{amsmath}
\usepackage{amssymb}
\theoremstyle{plain} 
\newtheorem{thmx}{Theorem}

\newtheorem{thm}{Theorem}[section]
\newtheorem{lemma}[thm]{Lemma}
\newtheorem{prop}[thm]{Proposition}

\newtheorem{exmp}[thm]{Example}

\theoremstyle{definition} 
\newtheorem{defn}[thm]{Definition}

\theoremstyle{remark} 
\newtheorem{rmk}[thm]{Remark}
\numberwithin{equation}{section}

\DeclareMathOperator{\linspan}{span} 
\DeclareMathOperator{\Iso}{Iso} 
\DeclareMathOperator{\supp}{supp} 

\newcommand{\A}{\mathcal{A}}		
\newcommand{\C}{\mathbb{C}}			
\newcommand{\Z}{\mathbb{Z}}	
\newcommand{\N}{\mathbb{N}}			
\newcommand{\T}{\mathbb{T}}			

\newcommand{\G}{\mathcal{G}}		    
\newcommand{\Go}{\mathcal{G}^{(0)}}		

\def\Isoint#1{\Iso (#1)^\circ}

\title[$C^*$-uniqueness for Groupoids]{$C^*$-uniqueness Results for Groupoids}

\usepackage[foot]{amsaddr}
\author{Are Austad and Eduard Ortega}
\address{Norwegian University of Science and Technology, Department of Mathematical Sciences, Trondheim, Norway.}
\email{are.austad@ntnu.no, eduard.ortega@ntnu.no}
\begin{document}
	\maketitle

	\begin{abstract}
		For a second-countable locally compact Hausdorff \'etale groupoid $\G$ with a continuous $2$-cocycle $\sigma$ we find conditions that guarantee that $\ell^1 (\G,\sigma)$ has a unique $C^*$-norm.
	\end{abstract}
	
	\section{Introduction}
	Given a reduced (Banach) $*$-algebra $\A$, the enveloping $C^*$-algebra $C^* (\A)$ plays a fundamental role in the representation theory of $\A$. However, any faithful $*$-representation of $\A$ will yield a $C^*$-completion of $\A$, and one may ask if this completion is isomorphic to the enveloping $C^*$-algebra. In the particular case of a locally compact group $G$, we may for example consider the $*$-algebras $C_c (G)$ or $L^1 (G)$. There are then two canonical $C^*$-norms, namely the one arising from the left regular representation and the maximal $C^*$-norm. It is well-known that $G$ is an amenable group if and only if these two $C^*$-norms coincide. However, even for amenable groups we can not rule out that there are $C^*$-norms on $C_c (G)$ and $L^1 (G)$ that are properly dominated by the norm induced by the left regular representation. Examples of this are given in \cite[p.\ 230]{Boidol84}. This invites the notion of \emph{$C^*$-uniqueness}. A reduced $*$-algebra $\A$ is called $C^*$-unique if $C^* (\A)$ is the unique $C^*$-completion of $\A$ up to isomorphism. This was extensively studied in \cite{Barnes83} for $*$-algebras. Moreover, a more specialized study for convolution algebras of locally compact groups was conducted in \cite{Boidol84}, where $C^*$-uniqueness of $L^1 (G)$ was studied by considering properties of the underlying group $G$. These two papers spawned investigations on $C^*$-uniqueness in the following decades, see for example \cite{HaKuVo90, LeuNg, DeKa13, Austad20}. In later years, algebraic $C^*$-uniqueness of discrete groups has garnered some attention \cite{GrMuRo, AlKy19, Scarparo}. This is the study of $C^*$-uniqueness of the group ring $\C [\Gamma]$ for a discrete group $\Gamma$ and is not equivalent to the study of $C^*$-uniqueness of $\ell^1 (\Gamma)$, see \cref{rmk:C*-uniqueness-ell1-vs-group-ring}. 
	
	We will in this paper study the $C^*$-uniqueness of certain Banach $*$-algebras associated to groupoids. To be more precise, given a second-countable locally compact Hausdorff \'etale groupoid $\G$ with a normalized continuous $2$-cocycle $\sigma$, we will study the $C^*$-uniqueness of the $I$-norm completion of $C_c (\G,\sigma)$, which will be denoted by $\ell^1 (\G,\sigma)$, see \eqref{eq:I-norm}. Here $C_c (\G,\sigma)$ denotes the space $C_c (\G)$ equipped with $\sigma$-twisted convolution and involution, see \eqref{eq:twisted-conv} and \eqref{eq:twisted-inv}, and similarly for $\ell^1 (\G, \sigma)$. Associated to $\ell^1 (\G,\sigma)$ are two canonical $C^*$-norms, namely the one coming from the $\sigma$-twisted left regular representation, see \eqref{eq:2-twisted-left-reg-rep}, and the full $C^*$-norm. If these coincide we say $\G$ has the $\sigma$-weak containment property. The technicalities will be postponed to \cref{sec:prelim-groupoids}. Letting $\Isoint{\G}$ denote the interior of the isotropy subgroupoid of $\G$, we will first find that for $\ell^1 (\G,\sigma)$ to be $C^*$-unique, it is sufficient that $\ell^1 (\Isoint{\G},\sigma)$ is $C^*$-unique. If we further let $\Isoint{\G}_x$ denote the fiber of $\Isoint{G}$ in the point $x \in \Go$, and let $\sigma_x$ denote the restriction of $\sigma$ to this fiber, we have the following main result. 
	
	\begin{thmx}[cf.\ \cref{thm:main-thm}]\label{thm:intro1}
		Let $\G$ be a second-countable locally compact Hausdorff \'etale groupoid with a continuous $2$-cocycle $\sigma$. Suppose that $\G$ has the $\sigma$-weak containment property. Then
		$\ell^1 (\G,\sigma)$ is $C^*$-unique if all the twisted convolution algebras $\ell^1 (\Isoint{\G}_x,\sigma_x)$, $x \in \Go$, are $C^*$-unique. 	
	\end{thmx}
	
	The theorem allows us to deduce $C^*$-uniqueness of $\ell^1 (\G,\sigma)$ by considering $C^*$-uniqueness of the (twisted) convolution algebras of the discrete groups $\Isoint{\G}_x$, $x\in \Go$. The latter has been studied earlier, the untwisted case in \cite{Boidol84} and the twisted case in \cite{Austad20}. Using this we obtain several examples of groupoids $\G$ for which $\ell^1 (\G,\sigma)$ is $C^*$-unique in \cref{sec:Examples}. Additionally, we are able to deduce $C^*$-uniqueness of some wreath products using our groupoid approach, see \cref{exmp:wreath}. 
	
	We will proceed in the following manner. In \cref{sec:prelims} we will collect all results we will need regarding $C^*$-uniqueness of Banach $*$-algebras, $C^*$-algebra bundles, as well as cocycle-twisted convolution algebras associated to second-countable locally compact Hausdorff \'etale groupoids. In \cref{sec:main-sec} we first present our main theorem, \cref{thm:main-thm}. The remainder of the section will be dedicated to its proof. Lastly, in \cref{sec:Examples} we present examples of $C^*$-unique convolution algebras coming from groupoids, as well as deducing $C^*$-uniqueness of some wreath products.

	\section{Preliminaries}\label{sec:prelims}
	
	\subsection{$C^*$-uniqueness for Banach $*$-algebras}\label{sec:prelims-C*-uniqueness}
	A representation of a Banach $*$-algebra $\A$ is a $*$-homomorphism $\pi \colon \A \to B (\mathcal{H})$, where $B(\mathcal{H})$ are the bounded linear operators on a Hilbert space $\mathcal{H}$. We say $\A$ is \emph{reduced} if $\A_\mathcal{R}=\{a\in\A: \pi(a)=0\text{ for every }*-\text{representation } \pi \text{ of }\A\}=\{0\}$. All Banach $*$-algebras we consider in the sequel will be reduced. The \emph{enveloping $C^*$-algebra} of a reduced Banach $*$-algebra $\A$ is the unique $C^*$-algebra $C^*(\A)$ which admits the following universal property: there exists an injective $*$-homomorphism $\Phi:\A\to C^*(\A)$ with dense range so that for every $*$-representation $\pi:\A\to B(\mathcal{H})$, there exists a unique $*$-representation $\hat{\pi}:C^*(\A)\to B(\mathcal{H})$ so that $\pi=\hat{\pi}\circ \Phi$. In order to ease notation in the sequel we will identify $\A$ with the Banach $*$-subalgebra $\Phi (\A)$ of $C^* (\A)$ whenever it is natural to do so. The enveloping $C^*$-algebra of a Banach $*$-algebra always exists \cite[Section 10.1]{Palmer}.
	
	\begin{defn} Let $\A$ be a reduced Banach $*$-algebra. We say that $\A$ is \emph{$C^*$-unique} if the $C^*$-norm given by 
		$$\|a\|:=\sup\{\|\pi(a)\|: \pi:\A\to B(\mathcal{H}) \text{ is a $*$-representation} \}$$ 
		for every $a\in \A$, is the unique $C^*$-norm on $\A$. In other words, $\A$ is $C^*$-unique if $C^* (\A)$ is the unique $C^*$-completion of $\A$ up to isomorphism.
	\end{defn}
	
	We will make repeated use of the following result on $C^*$-uniqueness of Banach $*$-algebras, see \cite[Proposition 10.5.19]{Palmer}.
	\begin{prop}\label{prop:C*-uniqueness-intersection-property}
		Let $\A$ be a reduced Banach $*$-algebra with enveloping $C^*$-algebra $C^*(\A)$. Then $\A$ is $C^*$-unique if and only if for every nonzero two-sided closed ideal $I \vartriangleleft C^*(\A)$ we have $\A \cap I \neq \{0\}$.
	\end{prop}
	
	\subsection{$C^*$-algebra bundles}\label{sec:prelims-bundles}
	The notion of a $C_0 (X)$-algebra will be of importance in the proof of the main theorem. 
	 Hence we briefly revise some basic notions and results on $C_0 (X)$-algebras and $C^*$-bundles. 
	\begin{defn} Let $X$ be a locally compact Hausdorff space. A \emph{$C_0(X)$-algebra} is a  $C^*$-algebra $A$ together with a non-degenerate injection $\iota:C_0(X)\to \mathcal{Z}(M(A))$, where the latter denotes the center of the multiplier algebra of $A$.
	\end{defn}
	We shall also need to consider (upper semi-continuous) $C^*$-bundles.
	\begin{defn}
		Let $X$ be a locally compact Hausdorff space and let $\{B_x\}_{x∈X}$ be a family of $C^*$-algebras. A map $f$ defined on $X$ such that $f(x)\in B_x$ for all $x\in X$, is called a \emph{section}. An \emph{upper semi-continuous $C^*$-bundle $\mathbf{B}$ over $X$} is a
		triple $(X, \{B_x\}_{x\in X}, \Gamma_0(\mathbf{B}))$, where $\Gamma_0(\mathbf{B})$ is a family of sections, such that the
		following conditions are satisfied:
		\begin{enumerate}
			\item $\Gamma_0(\mathbf{B})$ is a $C^*$-algebra under pointwise operations and supremum norm,
			\item for each $x\in X$, $B_x=\{f(x): f\in \Gamma_0(\mathbf{B}) \}$,
			\item for each $f\in \Gamma_0(\mathbf{B})$ and each $\varepsilon>0$, $\{x\in X: \vert f(x)\vert \geq \varepsilon\}$ is compact,
			\item $\Gamma_0(\mathbf{B})$ is closed under multiplication by $C_0(X)$, that is, for each $g\in C_0(X)$ and $f\in \Gamma_0(\mathbf{B})$, the section $gf$ defined by $gf(x)=g(x)f(x)$ is in $\Gamma_0(\mathbf{B})$. 
		\end{enumerate}
	\end{defn}
	
	The two above concepts can be combined to obtain the main theorem of \cite{Nilsen} which we present shortly for the reader's convenience. Suppose $X$ is a locally compact Hausdorff space, and suppose $A$ is a $C_0 (X)$-algebra with map $\iota \colon C_0 (X) \to \mathcal{Z}(M(A))$. For $x \in X$, denote by $J_x := C_0 (X\setminus \{x\})$ and realize $J_x \subseteq C_0 (X)$ in the natural way. Moreover, we define  $I_x :=  \iota (J_x) A$, which is a closed two-sided ideal of $A$. We then have the following result which will play a major role in the proof of \cref{thm:main-thm}.
	\begin{prop}[{\cite[Theorem 2.3]{Nilsen}}]\label{prop:nilsen-result}
		Let $X$ be a locally compact Hausdorff space and let $A$ be a $C_0 (X)$-algebra. Then there exists a unique upper semi-continuous $C^*$-bundle $\mathbf{B}$ over $X$ such that
		\begin{enumerate}
			\item [i)] the fibers $B_x = A /I_x$, and 
			\item [ii)] there is an isomorphism $\phi \colon A \to \Gamma_0 (\mathbf{B})$ satisfying $\phi (a)(x) = a + I_x$.
		\end{enumerate}
	\end{prop}

	\subsection{Groupoids, cocycle twists and associated algebras}\label{sec:prelim-groupoids}
	Given a groupoid $\G$ we will denote by $\G^{(0)}$ its unit space and write $r,s:\G\to \G^{(0)}$ for the range and source maps, respectively. We will also denote by $\G^{(2)}=\{(\alpha,\beta)\in \G\times \G: s(\alpha)=r(\beta) \}$ the set of \emph{composable elements}. In this paper, we will only consider groupoids $\G$ equipped with a second-countable locally compact Hausdorff topology making all the structure maps continuous. A groupoid $\G$ is called \emph{étale} if the range map, and hence also the source map, is a local homeomorphism. A subset $B$ of an \'etale groupoid $\G$ is called a \emph{bisection} if there is an open set $U\subseteq \G$ containing $B$ such that $r\colon U\to r(U)$ and $s\colon U\to s(U)$ are homeomorphisms onto open subsets of $\Go$. Second-countable locally compact Hausdorff \'etale groupoids have countable bases consisting of open bisections.
	
	Given $x\in \G^{(0)}$ we define by $\G_x:=\{\gamma\in \G: s(\gamma)=x\}$ and $\G^x:=\{\gamma\in \G: r(\gamma)=x\}$. Observe that if $\G$ is \'etale the sets $\G_x$ and $\G^x$ are discrete for every $x\in \G^{(0)}$. The \emph{isotropy group of $x$} is given by $\G_x^x:=\G^x\cap \G_x=\{\gamma\in \G: s(\gamma)=r(\gamma)=x\}$, and the \emph{isotropy subgroupoid of $\G$} is the subgroupoid $\Iso(\G):=\bigcup_{x\in \G^{(0)}}\G_x^x$ with the relative topology from $\G$. Let $\Isoint{G}$ denote the interior of $\Iso(\G)$. We then say that $\G$ is \emph{topologically principal} if $\Isoint{\G}=\G^{(0)}$. 
	
	We will consider groupoid twists where the twist is implemented by a normalized continuous $2$-cocycle. To be more precise, let $\G$ be a second countable locally compact \'etale groupoid. A \emph{normalized continuous $2$-cocycle} is then a continuous map $\sigma \colon \G^{(2)} \to \T$ satisfying
	\begin{equation}\label{eq:2-cocycle-unit-condition}
	\sigma (r(\gamma),\gamma) = 1 = \sigma (\gamma, s(\gamma))
	\end{equation}
	for all $\gamma \in \G$, and
	\begin{equation}\label{eq:2-cocycle-associativity-condition}
	\sigma (\alpha,\beta) \sigma (\alpha\beta,\gamma) = \sigma(\beta,\gamma) \sigma(\alpha,\beta\gamma)
	\end{equation}
	whenever $(\alpha,\beta),(\beta,\gamma)\in \G^{(2)}$. The set of normalized continuous $2$-cocycles on $\G$ will be denoted $Z^2 (\G,\T)$. Note that this is not the most general notion of a twist of a groupoid (see \cite[Chapter 5]{Sims}).
	
	Let $\G$ be a second-countable locally compact Hausdorff étale groupoid. We will define the $\sigma$-twisted convolution algebra $C_c (\G,\sigma)$ as follows: As a set it is just 
	$$C_c (\G,\sigma)=\{f:\G\to \C: f \text{ is continuous with compact support}\},$$ 
	but equipped with $\sigma$-twisted convolution product
	\begin{equation}\label{eq:twisted-conv}
	(f *_\sigma g) (\gamma) = 
	\sum_{\mu \in \G_{s(\gamma)}} f(\gamma\mu^{-1}) g(\mu) \sigma (\gamma\mu^{-1},\mu), \quad \text{$f,g\in C_c (\G,\sigma)$, $\gamma \in \G$,}
	\end{equation}
	and $\sigma$-twisted involution
	\begin{equation}\label{eq:twisted-inv}
	f^{*_\sigma}(\gamma) = \overline{\sigma(\gamma^{-1},\gamma)} \overline{f(\gamma^{-1})}, \quad \text{$f \in C_c (\G,\sigma)$, $\gamma \in \G$.}
	\end{equation}
	We complete $C_c (\G,\sigma)$ in the ''fiberwise $1$-norm'', also known as the $I$-norm, given by
	\begin{equation}\label{eq:I-norm}
	\Vert f \Vert_I = \sup_{x \in \Go}\max \{ \sum_{\gamma \in \G_x} \vert f(\gamma) \vert , \sum_{\gamma \in \G^{x}} \vert f(\gamma)\vert   \}
	\end{equation}
	for $f \in C_c (\G,\sigma)$. Denote by $\ell^1 (\G,\sigma)$ the completion of $C_c (\G,\sigma)$ with respect to the $I$-norm. This is a Banach $*$-algebra with the natural extensions of \eqref{eq:twisted-conv} and \eqref{eq:twisted-inv}. For later use we record the following lemma.
	\begin{lemma}\label{lemma:ell-1-cont-lemma}
	Let $\G$ be a second-countable locally compact Hausdorff \'etale groupoid. Then for any $f \in \ell^1 (\G)$, the map defined by 
	\begin{equation}\label{eq:ell-1-assignment}
	    \begin{split}
	        \Go \ni x \mapsto \max\{\sum_{\gamma \in \G_x} \vert f(\gamma)\vert ,\sum_{\gamma \in \G^x} \vert f(\gamma)\vert \},
	    \end{split}
	\end{equation}
	is continuous.
	\end{lemma}
	\begin{proof}
	    By density it is enough to show this for $f \in C_c (\G)$. It is well-known that $C_c (\G) = \linspan \{g\in C_c (\G) : g \text{ is supported on a bisection}\}$. Hence we may assume $f$ is supported on a bisection $U$, i.e.\ $\supp(f) \subseteq U$. Furthermore, for $f$ we denote the assignment of \eqref{eq:ell-1-assignment} by $F$. We thus wish to show that $F \in C(\Go)$.
	    
	    To this end, fix $x \in \Go$. As $f(x) = 0$ if $x\not\in s(U)$, we assume $x \in s(U)$. Since $s(x) = x$ and $s \colon U \to s(U)$ is a homeomorphism, we therefore have $x \in U$.  
	    Moreover, let $(x_i)_i \subseteq \Go$ be such that $x_i \to x$. Then eventually $x_i \in s(U)$ for all $i$ large enough. For such $i$ we have $F (x_i) = \vert f(\gamma_i)\vert$, where $\gamma_i$ is the unique element of $U$ with $s (\gamma_i)=x_i$. Now, as $s \colon U \to s(U)$ is a homeomorphism and $x_i \to x$, we have $\gamma_i \to \gamma \in U$, where $\gamma$ is the unique element of $U$ such that $s (\gamma) = x$. As $f \in C_c (\G)$, it follows that $f (\gamma_i)\to f(\gamma)$, and hence $F (x_i) \to F(x)$. Hence $F\in C(\Go)$, and the result follows. 
	\end{proof}

	We wish to understand when $\ell^1 (\G,\sigma)$ is $C^*$-unique, i.e.\ when it only permits one separating $C^*$-norm. To do this it will be of importance to use \cref{prop:C*-uniqueness-intersection-property}. 
	
	The \emph{(full) twisted groupoid $C^*$-algebra} $C^*(\G,\sigma)$ is the completion of $C_c (\G,\sigma)$ in the norm
	\begin{equation}
	\Vert f \Vert := \sup \{ \Vert \pi (f) \Vert : \text{$\pi$ is an $I$-norm bounded $*$-representation}\},
	\end{equation}
	for $f \in C_c (\G,\sigma)$.
	It was observed in \cite[Lemma 3.3.19]{Armstrong} that if $\G$ is étale, then every $*$-representation of $C_c(\G,\sigma)$ is bounded by the $I$-norm. 
	Then, since we are completing with respect to a supremum over $*$-representations, $C^*(\G,\sigma)$ is just the $C^*$-envelope of $\ell^1 (\G,\sigma)$. 
	
	Now we will construct a faithful representation of $\ell^1(\G,\sigma)$ called the \emph{$\sigma$-twisted left regular representation}. In particular, we have that $\ell^1(\G,\sigma)$ is reduced. The completion of the image of $\ell^1(\G,\sigma)$ under the left regular representation is called the \emph{$\sigma$-twisted reduced groupoid $C^*$-algebra of $\G$} and will be denoted $C^*_r (\G,\sigma)$. Let $x \in \Go$. Then there is a representation $L^{\sigma,x} \colon C_c (\G,\sigma) \to B(\ell^2 (\G_x))$ which is given by
	\begin{equation}\label{eq:2-twisted-left-reg-rep}
	L^{\sigma,x} (f) \delta_\gamma = \sum_{\mu \in \G^{r(\gamma)}} \sigma (\mu,\mu^{-1}\gamma)f(\mu) \delta_{\mu \gamma}, \quad \text{for $f \in C_c (\G,\sigma)$ and $\gamma \in \G_x$}.
	\end{equation}
	We then obtain a faithful $I$-norm bounded $*$-representation of $C_c (\G,\sigma)$ given by
	\begin{equation}
	\bigoplus_{x\in \Go} L^{\sigma,x} \colon C_c(\G,\sigma) \to \bigoplus_{x\in \Go} B(\ell^2 (\G_x)) \subset B(\bigoplus_{x\in \Go} \ell^2 (\G_x)).
	\end{equation}
	$C^*_r (\G,\sigma)$ is then the completion of the image of  $C_c (\G,\sigma)$ under the left regular representation. As the $*$-representation is $I$-norm bounded, $C^*_r (\G,\sigma)$ is also the completion of $\ell^1 (\G,\sigma)$ in the same norm. Therefore, since $C^*(\G,\sigma)$ is the $C^*$-envelope of $\ell^1(\G,\sigma)$, by universality, there exist a natural (surjective) $*$-homomorphism $\lambda:C^*(\G,\sigma)\to C^*_r(\G,\sigma)$. 
	

	\begin{defn}
		Let $\G$ be a second-countable locally compact Hausdorff groupoid and let $\sigma \in Z^2(G,\T)$. We say that $\G$ has the \emph{$\sigma$-weak containment property} when the natural map $\lambda \colon C^* (\G,\sigma)\to C^*_r (\G,\sigma)$ is an isomorphism.
	\end{defn}
	
	If $\G$ is an amenable groupoid \cite{AnDRen}, we have that  $C^*_r (\G,\sigma) = C^* (\G,\sigma)$ for every $\sigma \in Z^2 (\G, \T)$ \cite[Proposition 6.1.8]{AnDRen}, and hence $\G$ has the $\sigma$-weak containment property for every $\sigma \in Z^2 (\G, \T)$. In \cite{Will} it was proved that amenability is not equivalent to having the weak containment property. On the other hand, it is not known to the authors whether the $1$-weak containment property (where $1$ denotes the trivial twist) is equivalent to the $\sigma$-weak containment property for every  $\sigma \in Z^2 (\G, \T)$.


	\begin{rmk}\label{rmk:C*-uniqueness-ell1-vs-group-ring}
		While both $\ell^1 (\G,\sigma)$ and $C_c (\G,\sigma)$ complete to the same $C^*$-algebras $C^* (\G,\sigma)$ and $C^*_r (\G,\sigma)$ in the above setup, the question of $C^*$-uniqueness of $\ell^1 (\G,\sigma)$ is not equivalent to $C^*$-uniqueness of the $*$-algebra $C_c (\G,\sigma)$. To see this, let $\G = \Z$, the group of integers and consider the trivial twist $\sigma=1$. Then $\ell^1 (\Z,1) = \ell^1 (\Z)$ is $C^*$-unique by \cite{Boidol80}, while $C_c(\Z) = \C [\Z]$ is not $C^*$-unique by \cite[Proposition 2.4]{AlKy19}.
	\end{rmk}

	Denoting the restriction of $\sigma$ to $\Isoint{G}\subseteq \G$ also by $\sigma$, we define the Banach $*$-subalgebra  $\ell^1(\Isoint{\G},\sigma)$  of $\ell^1(\G,\sigma)$.  We then have the following result.  
	\begin{prop}[{\cite[Proposition 5.3.1]{Armstrong}}]\label{prop:first-armstrong-result}
		Let $\G$ be a second-countable locally compact Hausdorff \'etale groupoid and $\sigma \in Z^2 (\G,\T)$. There is a $*$-homomorphism 
		\begin{equation*}
		\iota \colon C^* (\Isoint{\G},\sigma) \to C^* (\G,\sigma)
		\end{equation*}
		such that
		\begin{equation*}
		\iota(f)(\gamma) = 
		\begin{cases}
		f(\gamma) & \text{if $\gamma \in \Isoint{\G}$,}\\
		0 & \text{otherwise,}
		\end{cases}
		\end{equation*}
		for all $f \in C_c (\Isoint{\G}, \sigma)$. This homomorphism descends to an injective $*$-homomorphism 
		\begin{equation*}
		\iota_r \colon C^*_r (\Isoint{\G},\sigma) \to C^*_r (\G,\sigma).
		\end{equation*}
	\end{prop}

	We observe that the homomorphism $\iota$ is an isometry at the $\ell^1$-level, i.e.\ that $\iota \colon \ell^1 (\Isoint{\G},\sigma) \to \ell^1 (\G,\sigma)$ is an isometric $*$-homomorphism. 
	
	We then also have the following result from \cite{Armstrong}  which will be key to our approach to study $C^*$-uniqueness of twisted groupoid convolution algebras in \cref{sec:main-sec}.
	
	\begin{prop}[{\cite[Theorem 5.3.13]{Armstrong}}]\label{prop:twisted-injective-iff-injective-on-subalg}
		Let $\G$ be a second-countable locally compact Hausdorff \'etale groupoid and let $\sigma \in Z^2 (\G, \T)$. Let $\iota_r \colon C^*_r (\Isoint{\G}, \sigma) \to C^*_r (\G,\sigma)$ be the injective $*$-homomorphism of \cref{prop:first-armstrong-result}. 
		Suppose $A$ is a $C^*$-algebra and that $\Psi \colon C^*_r (\G,\sigma) \to A$ is a homomorphism. Then $\Psi$ is injective if and only if $\Psi \circ \iota_r \colon C^*_r (\Isoint{\G},\sigma)\to A$ is and injective homomorphism.
	\end{prop}

	\section{$C^*$-uniqueness for cocycle-twisted groupoid convolution algebras}\label{sec:main-sec}
	We begin this section by presenting our main theorem. The remainder of the section will be dedicated to proving it. 
	
	Given a second-countable locally compact Hausdorff \'etale groupoid $\G$ and $\sigma \in Z^2 (\G,\T)$, denote the restriction of $\sigma$ to the fiber $\Isoint{\G}_x$ by $\sigma_x$. Note that $\sigma_x$ is continuous as $\Isoint{\G}_x$ is discrete, i.e.\ $\sigma_x \in Z^2 (\Isoint{\G}_x,\T)$. The following then constitutes our main theorem.
	
	\begin{thm}\label{thm:main-thm}
		Let $\G$ be a second-countable locally compact Hausdorff \'etale groupoid and $\sigma \in Z^2 (\G,\T)$. Suppose that $\G$ has the $\sigma$-weak containment property. Then 
		$\ell^1 (\G,\sigma)$ is $C^*$-unique if all the twisted convolution algebras $\ell^1 (\Isoint{\G}_x,\sigma_x)$, $x \in \Go$, are $C^*$-unique. 
	\end{thm}
	
	As a first step towards proving \cref{thm:main-thm} we relate $C^*$-uniqueness of $\ell^1 (\G,\sigma)$ to $C^*$-uniqueness of $\ell^1 (\Isoint{\G},\sigma)$.
	
	\begin{prop}\label{prop:twisted-Iso-C*unique-gives-G-C*unique} Suppose $\G$ is a second-countable locally compact  Hausdorff \'etale groupoid with the $\sigma$-weak containment property for $\sigma \in Z^2 (\G,\T)$. 
	If $\ell^1 (\Isoint{\G},\sigma)$ is $C^*$-unique, then $\ell^1 (\G,\sigma)$ is $C^*$-unique.
	\end{prop}
	\begin{proof}
		Suppose $\ell^1 (\Isoint{\G},\sigma)$ is $C^*$-unique. Then in particular  $C^*(\Isoint{\G},\sigma)=C^*_r(\Isoint{\G},\sigma)$. Let $\{0\} \neq J \vartriangleleft C^* (\G,\sigma)=C^*_r (\G,\sigma)$ be a closed two-sided ideal. By \cref{prop:C*-uniqueness-intersection-property} it suffices to show that $J\cap \ell^1 (\G,\sigma) \neq \{0\}$. By \cref{prop:twisted-injective-iff-injective-on-subalg} we have $C^* (\Isoint{\G},\sigma)\cap J \neq \{0\}$ as the $*$-homomorphism $C^* (\G,\sigma) \to C^*(\G,\sigma)/J$ is not injective. Now define $I := J \cap C^* (\Isoint{\G},\sigma)$. It is straightforward to verify that $I$ is a two-sided ideal in $C^* (\Isoint{\G},\sigma)$, and as both $J$ and $C^* (\Isoint{\G},\sigma)$ are closed in $C^* (\G,\sigma)$, $I$ is also closed in $C^* (\Isoint{\G},\sigma)$. By $C^*$-uniqueness of $\ell^1 (\Isoint{\G},\sigma)$ it then follows that $I \cap \ell^1 (\Isoint{\G},\sigma) \neq \{0\}$. From this we get
		\begin{equation*}
		\{0\} \neq I \cap \ell^1 (\Isoint{\G},\sigma) = J \cap \ell^1 (\Isoint{\G},\sigma) \subset J \cap \ell^1 (\G,\sigma),
		\end{equation*}
		from which we deduce by \cref{prop:C*-uniqueness-intersection-property} that $\ell^1 (\G,\sigma)$ is $C^*$-unique. 
	\end{proof}
	
	Having related the question of $C^*$-uniqueness of $\ell^1 (\G,\sigma)$ to a question regarding $C^*$-uniqueness of $\ell^1 (\Isoint{\G},\sigma)$, we proceed to further relate this to $C^*$-uniqueness of $\ell^1 (\Isoint{\G}_x,\sigma_x)$ for $x\in \Go$. To do this we will show that for any $*$-representation $\pi \colon \ell^1 (\Isoint{\G},\sigma) \to B(\mathcal{H})$, the resulting $C^*$-algebra $C^*_\pi (\Isoint{\G},\sigma)$ is a $C_0 (\Go)$-algebra. This is the content of \cref{lemma:Cpi-is-C0-alg}. However, we first do some preparatory work.
	
	First observe that there exists a $*$-homomorphism $\phi:C_0(\Go)\to \mathcal{Z}(\ell^1(\Isoint{\G},\sigma))$, the latter meaning the center of $\ell^1(\Isoint{\G},\sigma)$. Indeed, as $\Go$ is open in $\Isoint{\G}$, we may take $\phi$ to be the inclusion where we extend functions in $C_0 (\Go)$ by zero. The map $\phi$ is clearly isometric. 
	As $\phi$ can be viewed as an inclusion, we omit writing it from now on to ease notation. Then given $g\in C_0(\Go)$ and $f\in C_c(\Isoint{\G},\sigma)$ we have that 
	
	\begin{equation*}
	\begin{split}
	(g*_\sigma f)(\gamma) & = g(r(\gamma))f(\gamma)\sigma(r(\gamma),\gamma)=g(r(\gamma))f(\gamma)\\
	& =f(\gamma)g(s(\gamma))\sigma(\gamma, s(\gamma))=(f*_\sigma g)(\gamma)\,,
	\end{split}
	\end{equation*}
	for every $\gamma\in \Isoint{\G}$. The resulting action of $C_0 (\Go)$ on $\ell^1(\Isoint{\G},\sigma)$ can then be viewed as pointwise multiplication in the fibers of $\Go$. By continuity we can extend $\phi$ to a continuous $*$-homomorphism from $C_0 (\Go)$ to $\mathcal{Z}(\ell^1(\Isoint{\G},\sigma))$. 
	Let $\pi:\ell^1 (\Isoint{\G},\sigma) \to B(\mathcal{H})$ be a $*$-representation and let $C^*_\pi(\Isoint{\G},\sigma)$ denote the completion 
	in the operator norm of $B(\mathcal{H})$. Define the map 
	$\iota:=\pi\circ \phi:C_0(\Go)\to \pi(\mathcal{Z}(\ell^1(\Isoint{\G},\sigma)))$. We have that $$\pi(\mathcal{Z}(\ell^1(\Isoint{\G},\sigma)))=\mathcal{Z}(\pi(\ell^1(\Isoint{\G},\sigma)))\subseteq \mathcal{Z}(M(C^*_\pi(\Isoint{\G},\sigma)))\,.$$
	The following is then immediate.
	\begin{lemma}\label{lemma:Cpi-is-C0-alg}
		Let $\G$ be a second-countable locally compact Hausdorff \'etale groupoid and $\sigma \in Z^2 (\G,\T)$. Let $\pi$ be a $*$-representation of $\ell^1(\Isoint{\G},\sigma)$. Then $C^*_\pi(\Isoint{\G},\sigma)$ is a $C_0(\Go)$-algebra.
	\end{lemma} 
	
	Now fix $x \in \Go$ and denote by $J_x=C_0 (\Go \setminus \{x\})$ the space of continuous functions of $\Go$ vanishing at $x$. As $C_0 (\Go)$ is central in $\ell^1 (\Isoint{\G},\sigma)$ and $J_x$ is a closed two-sided ideal of $C_0 (\Go)$, the space $I_x := J_x \cdot \ell^1 (\Isoint{\G},\sigma)$ is a closed two-sided ideal in $\ell^1 (\Isoint{\G},\sigma)$. Recall that we denote by $\sigma_x$ the restriction of $\sigma$ to the fiber $\Isoint{\G}_x$. We then have the following result. 
	
	\begin{lemma}\label{lemma:Banach-alg-fiber-iso} Let $\G$ be a second-countable locally compact Hausdorff \'etale groupoid and let $\sigma \in Z^2 (\G,\T)$.
		For every $x\in \Go$ 
		the map $\psi_x \colon \ell^1 (\Isoint{\G}, \sigma) \to \ell^1 (\Isoint{\G}_x,\sigma_x)$ given by restriction of functions is a continuous $*$-homomorphism inducing an isometric $*$-isomorphism between $\ell^1 (\Isoint{\G},\sigma)/I_x$ and $\ell^1 (\Isoint{\G}_x, \sigma_x)$.
	\end{lemma}
	
	\begin{proof}
		For $f\in C_c(\Isoint{\G},\sigma)$ we have
		\begin{equation*}
		\Vert \psi_x (f) \Vert_{ \ell^1 (\Isoint{\G}_x)} = \sum_{\gamma \in \Isoint{\G}_x} \vert f (\gamma)\vert \leq \sup_{y \in \Go} \sum_{\mu \in \Isoint{\G}_y} \vert f (\mu)\vert = \Vert f \Vert_I
		\end{equation*}
		for all $f \in C_c(\Isoint{\G},\sigma)$.
		Thus $\psi_x$ is a $I$-norm decreasing map, so it extends to a continuous $*$-homomorphism $\psi_x:\ell^1(\Isoint{\G},\sigma) \to \ell^1(\Isoint{\G}_x,\sigma_x)$. It is surjective by Tietze's extension theorem. 
		
		Next we want to show that $\ker \psi_x=I_x$. First observe that given $g\in C_0(\Go)$ and $h\in C_c(\G,\sigma)$ we have that 
		\begin{equation*}
		\begin{split}
		\psi_x (g *_\sigma h) (\gamma) &= (\psi_x(g)*_\sigma \psi_x(h)) (\gamma) = \sum_{\mu \in \Isoint{\G}_x} g (\mu) h (\mu^{-1}\gamma)\sigma(\mu,\mu^{-1}\gamma)  \\
		&  = g (x) h (x\gamma)\sigma(x,\gamma) =g(x)h(\gamma)\,,
		\end{split}
		\end{equation*}
		for every $\gamma\in \Isoint{\G}_x$.
		
		Now let $f \in I_x$. We may then assume that $f$ is the norm limit of elements $f_n$ of the form  $f_n = \sum_{i=1}^n g_i *_\sigma h_i$, where $g_i \in J_x$ and $h_i \in C_c (\Isoint{\G},\sigma)$ for all $i \in \N$. It suffices to prove that $\psi_{x}(g_i *_\sigma h_i)=0$ for all $i \in \N$. For any $\gamma \in \Isoint{\G}_x$ we then have $\psi_x (g_i *_\sigma h_i) (\gamma)=g_i(x)h_i(\gamma)=0$ since $g_i(x)=0$. Then it follows that $\psi_x (f_n) = 0$ for every $n\in \N$, and by continuity $\psi_x(f)=0$. Thus,  $I_x \subset \ker \psi_x$.
		
		Conversely, suppose $f \in \ker \psi$. Then  $f=\lim f_n$ for some $f_n\in C_c(\G,\sigma)\cap \ker \psi_x$, and hence $f_n(x)=0$ for every $n\in \N$. Let $\{\rho_\lambda\}_{\lambda\in \Lambda}\subset C_0(\Go\setminus \{x\})$ be a partition of the unit of $\Go\setminus \{x\}$. Then  given $n\in \N$ there exists a finite subset $\Lambda_n$ of $\Lambda$, such that $g_n:=\sum_{\lambda\in \Lambda_n} \rho_n\in  C_0(\Go\setminus \{x\})=J_x$ and $g_n(y)=1$ for every $y\in r(\supp(f_n))=s(\supp(f_n))$, and hence $$f_n(\gamma)=g_n(r(\gamma))f_n(\gamma)\sigma(r(\gamma),\gamma)=(g_n*_\sigma f_n)(\gamma)$$ for every $\gamma \in \G$. Therefore we have that
		\begin{equation*}
		f = \lim_{n\to \infty} f_n= \lim_{n\to \infty} (g_n*_\sigma f_n)  \in  \overline{J_x \cdot \ell^1 (\Isoint{\G},\sigma)} = I_x, 
		\end{equation*}
		as we wanted. We would like to see that the isomorphism $\ell^1 (\Isoint{\G},\sigma)/I_x \cong \ell^1 (\Isoint{\G}_x,\sigma_x)$ is isometric. To do that, it is enough to check that 
		$$\inf \{\| f+h \|: h\in C_0(\Go\setminus \{x\})\cdot C_c(\G,\sigma)\}=\| \psi_x(f)\|$$
		for every $f\in C_c(\G,\sigma)$. Observe that by continuity of $\psi_x$ we have
		$\| f+h \|\geq \| \psi_x(f)\|$ for every $h\in C_0(\Go\setminus \{x\})\cdot C_c(\G,\sigma)$. 
		As $\G$ is second-countable locally compact Hausdorff, so is $\Go \setminus \{x\}$. Hence it is paracompact, and we can guarantee that there is a countable partition of unity $\{\rho_i \}_{i=1}^{\infty}$ for $\Go \setminus \{x\}$. 
		For $n\in \N$ let $U_n :=\Go\setminus \bigcup_{i=1}^n\supp(\rho_i)$. Then we have
		$$\|f-(\sum_{i=0}^n\rho_i)f\| \leq \max_{y\in U_n}\|\psi_y(f)\|\,.$$
		By \cref{lemma:ell-1-cont-lemma} the assignment $\G^{(0)} \ni x \mapsto \max\{\sum_{\gamma\in \G_x} |f(\gamma)|, \sum_{\gamma \in \G^x} |f(\gamma)|  \}$ is continuous. 
		It follows that for every $\varepsilon>0$ there exists $n$ such that $|\|\psi_y(f)\|-\|\psi_x(f)\||<\varepsilon$ for every $y\in U_n$. As $U_k \supset U_{k-1}$ for all $k$, it follows that $\|f-(\sum_{i=0}^k\rho_i)f\| \leq \Vert \psi_x (f)\Vert + \varepsilon$ for all $k \geq n$. As $\varepsilon$ was arbitrary, this finishes the proof.
	\end{proof}
	
	We may finally prove \cref{thm:main-thm}.
	
	\begin{proof}[Proof of \cref{thm:main-thm}]
		By \cref{prop:twisted-Iso-C*unique-gives-G-C*unique} it suffices to show that the condition implies that $\ell^1 (\Isoint{\G},\sigma)$ is $C^*$-unique. As above, denote by $J_x=C_0 (\Go \setminus \{x\})$ and by $I_x := \overline{J_x \cdot \ell^1 (\Isoint{\G},\sigma)}$ the resulting closed two-sided ideal in $\ell^1 (\Isoint{\G},\sigma)$. 
		Let $\pi \colon \ell^1 (\Isoint{\G},\sigma) \to B(\mathcal{H})$ be a faithful $*$-representation and denote by $C^*_\pi (\Isoint{\G},\sigma)$ the completion of $\pi(\ell^1 (\Isoint{\G},\sigma))$. Moreover, let $I_x^\pi$ denote the closure of $\pi(I_x)$ in $C^*_\pi (\Isoint{\G},\sigma)$. By \cref{prop:nilsen-result} and \cref{lemma:Cpi-is-C0-alg} there is an isomorphism $C^*_\pi (\Isoint{\G},\sigma)\cong \Gamma_0 (\mathbf{B}^\pi)$, where the fibers $B^\pi_x$, $x\in \Go$, are given by
		\begin{equation*}
		B^\pi_x = C^*_\pi (\Isoint{\G},\sigma)/I^\pi_x.
		\end{equation*}
		We will show that there is an injective $*$-homomorphism
		\begin{equation*}
		\Psi_x \colon \ell^1 (\Isoint{\G}_x,\sigma_x) \to B^\pi_x
		\end{equation*}
		for every $x\in \Go$. To do this, fix $x \in \Go$. 
		First, we show that the composition
		\begin{equation}
		\begin{split}
		\ell^1 (\Isoint{\G}_x,\sigma_x) \cong \ell^1 (\Isoint{\G},\sigma)/I_x \to  C^*_\pi(\Isoint{\G},\sigma) /I_x^\pi \cong B_x^\pi
		\end{split}
		\end{equation}
		given by first applying the isomorphism of \cref{lemma:Banach-alg-fiber-iso} and then applying the map $f + I_x \mapsto f + I_x^\pi$ for $f\in \ell^1 (\Isoint{\G},\sigma)$ is a well-defined continuous $*$-homomorphism. 
		This is our candidate for the map $\Psi_x$. Denote by $I^\pi_x$ also the image of the ideal $I^\pi_x \vartriangleleft C^*_\pi (\Isoint{\G},\sigma)$ in $\Gamma_0 (\mathbf{B}^\pi)$. It then suffices to show that if $F \in I^\pi_x$, then $F(x) = 0$. 
		
		To see this, note that we can let $C_0 (\Go\setminus \{x\})$ act on $C^*_\pi(\Isoint{\G}, \sigma)$ by pointwise multiplication to obtain a have a continuous $*$-homomorphism
		\begin{equation*}
		C_0 (\Go \setminus \{x\}) = J_x\to \mathcal{Z}(M(C^*_\pi(\Isoint{\G}, \sigma))),
		\end{equation*}
		which leaves $I^\pi_x$ invariant, and as a result $I^\pi_x$ becomes a Banach $J_x$-module. It is even non-degenerate as
		\begin{equation*}
		\begin{split}
		\overline{J_x I^\pi_x} = \overline{J_x \overline{J_x C^*_\pi(\Isoint{\G},\sigma)}} \supset \overline{J_x J_x C^*_\pi(\Isoint{\G},\sigma)} = \overline{J_x C^*_\pi(\Isoint{\G},\sigma)} = I^\pi_x,
		\end{split}
		\end{equation*}
		since $J_x$, being a $C^*$-algebra, has an approximate identity. It then follows by Cohen-Hewitt factorization that if $F \in I^\pi_x$, then $F = f\cdot H$, where $f \in J_x$ and $H \in I^\pi_x$. Then $F(x) = f(x) H(x) = 0$, and the map $\Psi_x$ is a well-defined $*$-homomorphism.
		
		As $\ell^1 (\Isoint{\G}, \sigma)$ is dense in its $C^*$-completion $C^* (\Isoint{\G},\sigma)$, it follows that the image of $\Psi_x$ is dense.
		
		Lastly, if $\Psi_x (f)=0$, then $\Psi_x (f) \in I^\pi_x$, and so $f\vert_{\Isoint{\G}_x}=0$ by the above argument. 
		Thus $\psi_x$ is injective. Hence we have a continuous dense embedding
		\begin{equation}
		\Psi_x\colon \ell^1 (\Isoint{\G}_x,\sigma_x) \hookrightarrow C^*_\pi (\Isoint{\G},\sigma)/J^\pi_x.
		\end{equation}
		Now $C^*_\pi (\Isoint{\G},\sigma)/J^\pi_x$ becomes a $C^*$-completion of $\ell^1 (\Isoint{\G}_x,\sigma_x)$. 
		Since $\pi$ is an arbitrary faithful $*$-representation of $\ell^1 (\Isoint{\G},\sigma)$, we deduce that this holds for all faithful $*$-representations. But as $\ell^1 (\Isoint{\G}_x,\sigma_x)$ is assumed $C^*$-unique, we may then deduce
		\begin{equation}\label{eq:C*-fiber-isos}
		C^*_\pi (\Isoint{\G},\sigma)/J^\pi_x \cong C^* (\Isoint{\G},\sigma)/J^{\mathrm{full}}_x,
		\end{equation}
		where $C^* (\Isoint{\G},\sigma)$ and $J^{\mathrm{full}}_x$ denotes the completions in the maximal $C^*$-norm. 
		As $x\in \Go$ was arbitrary, we deduce that this holds for all $x \in \Go$. Now let $B^{\mathrm{full}}_x = C^* (\Isoint{\G},\sigma)/J^{\mathrm{full}}_x$. By \cref{prop:nilsen-result} and \eqref{eq:C*-fiber-isos} we then have
		\begin{equation}
		C^*_\pi (\Isoint{\G},\sigma) \cong \Gamma_0 (\mathbf{B}^\pi) \cong \Gamma_0 (\mathbf{B}^{\mathrm{full}}) \cong C^*(\Isoint{\G},\sigma).
		\end{equation}
		From this we deduce that $\ell^1 (\Isoint{\G},\sigma)$, and hence also $\ell^1 (\G,\sigma)$, is $C^*$-unique. 
	\end{proof}

	\section{Examples}\label{sec:Examples}
	In this section we present some (classes of) examples of $C^*$-unique groupoids. Due to the nature of our main result, \cref{thm:main-thm}, our examples will draw upon previously proved results on $C^*$-uniqueness of locally compact groups. We begin with a class of examples in the case of trivial cocycle twists. 
	
	\begin{exmp}[The untwisted case]\label{exmp:untwisted-case}
		If we consider a second-countable locally compact Hausdorff \'etale groupoid $\G$ with the trivial $2$-cocycle $\sigma = 1$, then $C^*$-uniqueness of $\ell^1 (\G,1)=\ell^1 (\G)$ can by \cref{thm:main-thm} be deduced by $C^*$-uniqueness of the Banach $*$-algebras $\ell^1 (\Isoint{\G}_x,\sigma_x) = \ell^1 (\Isoint{\G}_x)$ for $x \in \Go$. $C^*$-uniqueness of untwisted convolution algebras has been studied before, and it is known that for a locally compact group $G$, the Banach $*$-algebra $\ell^1 (G)$ is $C^*$-unique if $G$ is a semidirect product of abelian groups, or a group where every compactly generated subgroup is of polynomial growth \cite[p.\ 224]{Boidol84}. Hence if for every $x\in \Go$ the discrete group $\Isoint{\G}_x$ is of one of these types, $\ell^1 (\G)$ will be $C^*$-unique.
	\end{exmp} 
	
	In the case of locally compact groups it is well-known that amenability of the group is equivalent to the group having the weak containment property. Indeed, amenability is even equivalent to the $\sigma$-weak containment property for all continuous $2$-cocycles $\sigma$ of the group. Moreover, it is easy to see that if a group is $C^*$-unique, then it is amenable. The converse is however not true \cite[p.\ 230]{Boidol84}. In stark contrast to the case of locally compact groups, the following example shows that groupoids can be $C^*$-unique without even being amenable. 
	
	\begin{exmp}[Non-amenable $C^*$-unique groupoid] In \cite[Theorem 2.7]{AlFi} the authors constructed a second-countable, locally compact, Hausdorff non-amenable \'etale groupoid $\G$ such that $\Isoint{\G}=\Go$ and $C^*_r(\G)=C^*(\G)$. Then since $\ell^1(\Isoint{\G})=C_0(\Go)\subseteq\ell^1(\G)$, we have by \cref{prop:twisted-injective-iff-injective-on-subalg} that  every nonzero two-sided ideal $I$ of $C^*(\G)$ has nonzero intersection with $C_0(\Go)$, and hence with $\ell^1(\G)$. Therefore by \cref{prop:C*-uniqueness-intersection-property}  we have that $\ell^1(\G)$ is $C^*$-unique.
		
		In this particular case we may also deduce $C^*$-uniqueness of $\ell^1 (\G)$ in another way. Namely, as $\Isoint{\G}=\Go$, we have that $\Isoint{\G}_x$ is the trivial group for every $x \in \Go$. Hence $\ell^1 (\Isoint{\G}_x)$ is $C^*$-unique by \cref{exmp:untwisted-case}. This argument of course carries over to any topologically principal groupoid. Indeed, this approach shows that whenever $\G$ is a second-countable, locally compact, Hausdorff topologically principal \'etale groupoid, then $\ell^1 (\G,\sigma)$ is $C^*$-unique for any $\sigma \in Z^2 (\G,\T)$. 
	\end{exmp}
	
	We also have classes of examples that includes more general cocycle twists.
	\begin{exmp}[The twisted case]
		Let $\G$ be a second-countable locally compact Hausdorff \'etale groupoid, and let $\sigma\in Z^2 (\G,\T)$. By \cref{thm:main-thm} $C^*$-uniqueness of $\ell^1 (\G,\sigma)$ can be deduced by $C^*$-uniqueness of the Banach $*$-algebras $\ell^1 (\Isoint{\G}_x, \sigma_x)$, for $x \in \Go$, where $\sigma_x$ as before denotes the restriction of $\sigma$ to $\Isoint{\G}_x$. 
		$C^*$-uniqueness of twisted convolution algebras of locally compact groups was studied in \cite{Austad20}. In \cite[Theorem 3.1]{Austad20} it was found that if $G$ is a locally compact group and $c \in Z^2 (G,\T)$, then $L^1 (G,c)$ is $C^*$-unique if $L^1 (G_c)$ is $C^*$-unique, where $G_c$ denotes the Mackey group associated to $G$ and $c$. As a topological space $G_c$ is just $G\times \T$, but the binary operation is given by
		\begin{equation*}
		(x,\tau)\cdot (y,\eta) = (xy, \tau \eta \overline{c(x,y)}).
		\end{equation*}
		Thus we may relate $C^*$-uniqueness of $\ell^1 (\Isoint{\G},\sigma_x)$ to $C^*$-uniqueness of $\ell^1 (\Isoint{\G}_{\sigma_x})$, where $\Isoint{\G}_{\sigma_x}$ denotes the Mackey group associated to $\Isoint{\G}_x$ and $\sigma_x$, and we deduce that $\ell^1 (\G,\sigma)$ is $C^*$-unique if $\ell^1 (\Isoint{\G}_{\sigma_x})$ is $C^*$-unique for every $x \in \Go$. This happens if, for example, $\Isoint{\G}_{\sigma_x}$ is a group of one of the types discussed in \cref{exmp:untwisted-case}.
	\end{exmp}
	
	In the following example we are able to deduce $C^*$-uniqueness of a locally compact group not of the form discussed in \cref{exmp:untwisted-case} by relating the question to $C^*$-uniqueness of a groupoid.
	
	\begin{exmp}[The wreath product]\label{exmp:wreath}
		Let $\Gamma$ denote the wreath product $H\wr G :=\left( \bigoplus_G H\right) \rtimes G$ where $H$ is a finite abelian group and where $G$ is a countable discrete amenable  group. We will show that $\ell^1 (\Gamma)$ is $C^*$-unique. 
		
		To do this, let $\G=X\rtimes_\varphi G$ be the transformation groupoid where  $X=\prod_G \hat{H}$, and $\varphi$ is the shift homeomorphism of $X$ by $G$. $\G$ is amenable since $G$ is amenable. Then we have that 
		$$C^*(\Gamma)\cong C^*(\bigoplus_G H)\rtimes_\varphi G\cong C(X)\rtimes_\varphi G\,.$$
		Now recall that by the Fourier transform $\ell^1(\bigoplus_G H)\cong A(X)$, where $A(X)$ is a dense subalgebra of $C(X)$. Indeed, it becomes a Banach $*$-subalgebra of $C(X)$ when equipped with the induced $\ell^1$-norm through the Fourier transform, and then the isomorphism is also an isometry. It also follows that $C(X)$ is the completion of $\ell^1(\bigoplus_G H)$ with respect to some $C^*$-norm. We have  that $\ell^1(\Gamma)\cong\ell^1(\ell^1\left( \bigoplus_G H\right),G)\cong \ell^1(A(X),G)$ (see for example \cite[Remark and Notation 2.4]{LeuNg}). Then there exists an isometric embedding  $\iota:\ell^1(A(X),G)\hookrightarrow \ell^1(\G)$ defined as follows. If $F\in \ell^1 (A(X),G)$, we define $\iota (F)$ to be
		\begin{equation*}
		\iota(F) (x,g) = \widehat{f_g}(x),
		\end{equation*}
		for $x \in X =\prod_G \hat{H}$ and $g \in G$, where $f_g$ is the unique element of $\ell^1 (\bigoplus_G H)$ with $\widehat{f_g} = F(g)$. Therefore by the isomorphisms $C^*(\ell^1(\Gamma))\cong C^*(\ell^1(A(X),G)) \cong C^*(\ell^1(\G))$ it would be enough to check that any nonzero two-sided ideal $I$ of $C^*(\G)$ has a non-trivial intersection with the image of $\ell^1(A(X),G)$ by the inclusion $\iota$. Observe that then $\ell^1(\bigoplus_G H)\subseteq \ell^1(\Gamma)$  can be identified with $\iota(A(X))$  in $C(X)\subseteq C^*(\G)$. The groupoid $\G$ is clearly topologically principal, and hence $\ell^1(\G)$ is $C^*$-unique. Moreover, for every  closed two-sided ideal $\{0\} \neq I\unlhd C^*(\G)$ we have that $\{0\} \neq J:=I\cap C(X)$ \cite[Theorem 4.1]{KaTo90}. But since $\bigoplus_G H$ is locally finite, then $\ell^1(\bigoplus_G H)$, and hence $A(X)$, are $C^*$-unique by \cite{GrMuRo}. Thus, $J\cap A(X)\neq \{0\}$, which further implies $J \cap \ell^1(A(X),G) \neq \{0\}$. It follows that $\ell^1 (\Gamma)$ is $C^*$-unique. 
	\end{exmp}

	\section*{Acknowledgements}
	The first author wishes to thank Petter Nyland for valuable discussions during the development of this article.

\end{document}